\documentclass[authoryear]{elsarticle}

\journal{Statistics and Probability Letters}

\usepackage{amsmath,amssymb,amsthm}
\usepackage{enumerate}

\usepackage{fancybox}
\newlength{\querylen}
\theoremstyle{plain}
\newtheorem{theorem}{Theorem}[section]

\newtheorem{proposition}[theorem]{Proposition}
\newtheorem{corollary}[theorem]{Corollary}

\theoremstyle{definition}
\newtheorem{definition}{Definition}[section]

\newtheorem{example}{Example}[section]

\theoremstyle{remark}
\newtheorem{remark}{Remark}[section]

\renewcommand{\P}{\mathbf{P}}
\newcommand{\Prob}[1]{\P\{#1\}}
\newcommand{\E}{\mathbf{E}}
\newcommand{\R}{\mathbb{R}}

\newcommand{\salg}{\mathfrak{F}}
\newcommand{\Sphere}{{\mathbb{S}}^{d-1}}

\DeclareMathOperator{\conv}{conv}

\begin{document}

\begin{frontmatter}

\title{Lift expectations of random sets} 

\author[mad]{Marc-Arthur Diaye}
\ead{marc-arthur.diaye@univ-paris1.fr}

\author[gk]{Gleb A. Koshevoy}
\ead{koshevoy@cemi.rssi.ru}

\author[im]{Ilya Molchanov\corref{mycorrespondingauthor}}
\cortext[mycorrespondingauthor]{Corresponding author}
\ead{ilya.molchanov@stat.unibe.ch}

\address[mad]{CES, University Paris 1 Pantheon-Sorbonne, France}

\address[gk]{The Institute for Information Transmission Problems,
Bolshoy Karetny 19, 127051 Moscow, Russia}

\address[im]{University of Bern,
 Institute of Mathematical Statistics and Actuarial Science,
 Alpeneggstrasse 22,
 3012 Bern,
 Switzerland}

\begin{abstract}
  It is known that the distribution of an integrable random vector
  $\xi$ in $\R^d$ is uniquely determined by a $(d+1)$-dimensional
  convex body called the lift zonoid of $\xi$. This concept is
  generalised to define the lift expectation of random convex
  bodies. However, the unique identification property of distributions
  is lost; it is shown that the lift expectation uniquely identifies
  only one-dimensional distributions of the support function, and so
  different random convex bodies may share the same lift
  expectation. The extent of this nonuniqueness is analysed and it is
  related to the identification of random convex functions
  using only their one-dimensional marginals. Applications to
  construction of depth-trimmed regions and partial ordering of random
  convex bodies are also mentioned.
\end{abstract}

\begin{keyword}
  random set \sep selection expectation \sep lift zonoid
  \sep support function \sep risk measure \sep outlier
  \MSC[2010] 62H11 \sep 60D05 
\end{keyword}

\date{\today}

\end{frontmatter}

\section{Introduction}
\label{sec:introduction}

Probability theory provides numerous ways of identifying distributions
of random variables. Mentioning two (rather nontraditional) examples,
the distribution of an integrable random variable $\xi$ is uniquely
determined by its stop-loss transform $\E(t+\xi)_+$, $t\in\R$, where
$x_+$ denotes the positive part of $x\in\R$. Furthermore, \cite{hoe53}
showed that the sequence $\E\max(\xi_1,\dots,\xi_n)$, $n\geq1$, built
from i.i.d. copies of $\xi$ uniquely determines the distribution of
$\xi$.

Extensions of these identification results to random vectors are of a 
geometric nature and rely on the concept of zonoids and zonotopes.
Zonotopes form an important family of polytopes, which are defined as
Minkowski (elementwise) sums of a finite number of
segments. \emph{Zonoids} are convex sets that appear as limits in the
Hausdorff metric of a sequence of zonotopes, see
\cite[Sec.~3.5]{schn2}. In the plane, all (centrally) symmetric convex
sets are zonoids, while the symmetry is only a strictly necessary
condition in dimensions three and more.

Zonoids can also be described as expectations of random segments. For
this purpose, recall that a \emph{random convex closed set} $X$ in
$\R^d$ is a map from a probability space $(\Omega,\salg,\P)$ to the
family of convex closed sets in $\R^d$, which is measurable in the
sense that $\{\omega:\; X(\omega)\cap K\neq\emptyset\}\in\salg$ for
all compact sets $K$ in $\R^d$, see \cite[Def.~1.1.1]{mo1}. If $X$ is
almost surely compact and non-empty, it is called a \emph{random
  convex body}. The measurability condition is then equivalent to the
fact that the \emph{support function} of $X$
\begin{displaymath}
  h_X(u)=\sup\{\langle x,u\rangle:\; x\in X\},\qquad u\in\R^d,
\end{displaymath}
is a random function of $u$, where $\langle x,u\rangle$ is the scalar
product.  The distribution of a random convex body is uniquely
identified by the finite-dimensional distributions of its support
function.

A random convex closed set $X$ is said to be \emph{integrable}, if
there exists an integrable random vector $\xi$ such that $\xi\in X$
a.s. This vector $\xi$ is called an integrable \emph{selection} of
$X$. The \emph{selection expectation} $\E X$ is 
the closure of the set of expectations of all its integrable
selections, see \cite[Sec.~2.1]{mo1}. The closure is not needed if 
\begin{displaymath}
  \|X\|=\sup\{\|x\|:\; x\in X\}
\end{displaymath}
is an integrable random variable. Then $X$ is said to be
\emph{integrably bounded}.

If $X$ is integrably bounded,
then $h_X(u)$ is integrable for all $u$, and
\begin{displaymath}
  \E h_X(u)=h_{\E X}(u),\qquad u\in\R^d.
\end{displaymath}
If $X$ is integrable and $\xi$ is its integrable selection, then
$h_X(u)=\langle\xi,u\rangle+h_{X-\xi}(u)$, whence $h_X(u)$ is either
integrable or has the well-defined expectation $+\infty$. 

If $\xi$ is a random vector in $\R^d$, then the \emph{segment}
$[\boldsymbol{0},\xi]$ with end-points being the origin $\boldsymbol{0}$ and $\xi$ is a random
convex body. This random convex body is integrable even for a
nonintegrable $\xi$, since it contains the origin. Its expectation
$Z_\xi=\E[\boldsymbol{0},\xi]$ is called the \emph{zonoid} of $\xi$, see
\citep{mos02} and \citep{mo1}. Often, symmetrised versions of zonoids
are defined as expectation of the segment $[-\xi,\xi]$ and assuming
the integrability of $\xi$, see \cite[Sec.~3.5]{schn2}.

The zonoid does not uniquely determine the distribution of
$\xi$, for instance, it does not change if $\xi$ is multiplied by an
independent nonnegative random variable with expectation one. The
extent of such nonuniqueness is explored by
\cite{mol:sch:stuc13}. Despite the nonuniqueness, the zonoid
delivers some information about the linear dependence between $\xi$
and $\eta$, see \cite{dal:scar03}. 

It is possible to achieve the uniqueness by uplifting $\xi$ into
$\R^{d+1}$. For this, consider the segment $[(0,\boldsymbol{0}),(1,\xi)]$ in
$\R^{d+1}$, and call $\E[(0,\boldsymbol{0}),(1,\xi)]=\hat{Z}_\xi$ the \emph{lift
  zonoid} of $\xi$, see \citep{kos:mos98} and \citep{mos02}. Since
\begin{displaymath}
  h_{\hat{Z}_\xi}(u_0,u) = \E (u_0+\langle \xi,u\rangle)_+,\quad (u_0,u)\in\R^{d+1},
\end{displaymath}
the support function of the lift zonoid is the stop-loss transform of
$\langle \xi,u\rangle$. Thus, the lift zonoid of $\xi$ determines
uniquely the distribution of the scalar products $\langle
\xi,u\rangle$ for all $u\in\R^d$ and so the distribution of
$\xi$. This fact goes back to \cite{har81}, was independently proved
by \cite{kos:mos98}, and further gave rise to numerous applications in
multivariate analysis, see \citep{mos02}.

This paper presents an extension of the lift zonoid concept for random
convex bodies. Section~\ref{sec:lift-zonoid-random} defines the
required lifting that gives rise to the corresponding expected
sets. In general, such a set is no longer a zonoid in the geometric
sense of \cite[Sec.~3.5]{schn2}, and so we call it \emph{lift
  expectation}. It is shown that the lift expectation of an integrable
random convex body $X$ characterises the distributions of the support
function of $X$ in any single direction $u$ such that $h_X(u)$ is
integrable. Equivalently, the lift expectation embodies the
information contained in the single marginals of a random convex
function.

Examples of lift expectations are provided in
Section~\ref{sec:prop-exampl-lift}.  
Section~\ref{sec:uniq-rand-sets} relates the uniqueness issue 
to the multivariate comonotonicity of the support
function. Section~\ref{sec:discr-distr} deals with random sets having
at most a finite number of realisations. Section~\ref{sec:mult-extens}
discusses an extension of this concept for $n$-tuples of random sets.
A numerical example is given in Section~\ref{sec:numerical-example}.

The range of possible applications of the lift expectation is similar
to those well-established for lift zonoids, see \citep{mos02}, e.g.,
to assessing the depth of set-valued observations and stochastic
ordering of random convex bodies. In particular, the lift expectation
can be used to identify outliers in samples of random convex bodies
--- such samples arise in applications to partially identified
problems in econometrics, see \citep{mol:mol18}.  A relation to risk
measures is mentioned in Example~\ref{ex:seg}.

\section{Lift expectation of a random set}
\label{sec:lift-zonoid-random}

\subsection{Univariate distributions of the support function}
\label{sec:defin-univ-distr}

Let $X$ be a random convex body in $\R^d$. Uplift it to $\R^{d+1}$ by
letting
\begin{equation}
  \label{eq:1}
  Y=\conv(\{0,\boldsymbol{0}\},\{1\}\times X)
\end{equation}
be the convex hull of the origin $(0,\boldsymbol{0})$ in $\R^{d+1}$
and the set $\{1\}\times X$.  The random convex body $Y$ is always
integrable; it is integrably bounded if and only if $X$ is integrably
bounded.

\begin{definition}
  \label{def:le}
  The set $\E Y$ (that is, the selection expectation of $Y$) is called
  the \emph{lift expectation} of $X$ and denoted by $\hat{Z}_X$.
\end{definition}

The following result establishes that the lift expectation provides 
exactly the same information as the distributions of the
support function $h_X(u)$ for any given $u$ such that $h_X(u)$ is
integrable. 

\begin{theorem}
  \label{thr:sf}
  Assume that $X$ and $X'$ are integrable random convex bodies. Then
  $\hat{Z}_X=\hat{Z}_{X'}$ if and only if the distributions of $h_X(u)$
  and $h_{X'}(u)$ coincide for all $u$ such that at least one of them is
  integrable. 
\end{theorem}
\begin{proof}
  Since the function $(u_0+t)_+$ is monotonically increasing in $t$,
  the support function of $Y$ given by \eqref{eq:1} is
  \begin{displaymath}
    h_Y(u_0,u)=\sup\big\{(u_0+\langle x,u\rangle)_+:\; x\in X\big\}
    =\big(u_0+h_X(u)\big)_+
  \end{displaymath}
  for $u_0\in\R$ and $u\in\R^d$.  By \cite[Th.~1.1]{har81}, the
  expected value of $h_Y(u_0,u)$ considered as a function of
  $u_0\in\R$ uniquely determines the distribution of $h_X(u)$ if
  $h_X(u)$ is integrable. In the
  other direction, the distribution of $h_X(u)$ uniquely determines
  the expected value of $h_Y(u_0,u)$ and so $\hat{Z}_X$.
\end{proof}

\begin{remark}
  \label{rem:non-unique}
  The lift expectation of an integrably bounded random convex body $X$
  uniquely determines the one-dimensional distributions of $h_X(u)$
  for all $u\in\R^d$. The joint distributions of the support function
  at different directions are not necessarily uniquely identified.
  The extent of the nonuniqueness for random convex bodies with the
  same lift expectation corresponds to the possibilities of choosing a
  random sublinear function on $\R^d$ with given one-dimensional
  marginals. Similar questions arise in the studies of stochastic
  processes which share one-dimensional marginals with a martingale or
  with the Brownian motion, see \cite{MR2808243}.
\end{remark}

\begin{corollary}
  \label{cor:expectation}
  If $X$ is integrable, then the lift expectation $\hat{Z}_X$
  uniquely determines $\E X$ and $\E X=\{x:\; (1,x)\in\E Y\}$.
\end{corollary}
\begin{proof}
  The lift expectation determines the distribution of $h_X(u)$ for all
  $u$ and so the expectation $\E h_X(u)=h_{\E X}(u)$.  
  The selections of $Y$ can
  be obtained as $(\eta,\eta\xi_\eta)$, where $\eta$ is a random
  variable with values in $[0,1]$ and $\xi_\eta$ is a selection of
  $\eta X$. Thus, the intersection of $\E Y$ with the hyperplane
  $\{(1,u):\; u\in\R^d\}$ arises as the set of
  $(\E\eta,\E(\eta\xi_\eta))$ with $\E\eta=1$ and so $\eta=1$ almost
  surely. This yields the statement, since $\xi$ is an arbitrary
  selection of $X$.
\end{proof}

\begin{remark}
  With each integrably bounded random convex body $X$, it is possible
  to associate a nested family of convex bodies
  $K_n=\E\conv(X_1,\dots,X_n)$, $n\geq1$, being the expectation of the
  convex hull of $n$ i.i.d.\ copies of $X$. 
  \cite{vit87} showed that this sequence uniquely identifies the distribution
  of $X$ if $X$ is a singleton. Since $h_{K_n}(u)$ is the expectation
  of the maximum of $n$ i.i.d.\ copies of $h_X(u)$, the distribution
  of $h_X(u)$ for any single $u$ is uniquely identified by the
  sequence $\{h_{K_n}(u),n\geq1\}$, see 
  \citep{hoe53}. Thus, the information delivered by the nested family
  $\{K_n,n\geq1\}$ is identical to the information recoverable from the lift
  expectation of $X$.
\end{remark}

\begin{remark}
  It is possible to generalise the definition of the lift expectation
  by replacing $Y$ with the convex hull of the origin in $\R^{m+d}$ and $K\times X$
  for a convex body $K$ in an auxiliary space $\R^m$.  Then
  $h_{K\times X}(u',u'') =(h_K(u')+h_X(u''))_+$ for $u'\in\R^m$ and
  $u''\in\R^d$.  Therefore, such a generalised lift expectation still
  identifies only the one-dimensional marginals of $h_X$.
\end{remark}

\begin{remark}
  Each convex (not necessarily homogeneous) random function
  $\zeta(x)$, $x\in\R^d$, yields the support function of a convex body
  in $\R^{d+1}$ by letting
  \begin{displaymath}
    \tilde\zeta(t,x)=
    \begin{cases}
      t\zeta(x/t), & t>0,\\
      0, & \text{otherwise}, 
    \end{cases}
  \end{displaymath}
  which is called the perspective transform of $f$, see
  \citep{hir:lem93}. Thus, the one-dimensional marginal distributions
  of a random convex function can be identified by a convex set in
  $\R^{d+2}$, which is the lift expectation of the random set with the
  support function $\tilde\zeta(t,x)$. The lift
  expectation is a convex body in $\R^{d+2}$, which summarises all
  one-dimensional marginals of a random convex function.
\end{remark}

\subsection{Convexity for Minkowski sums and detection of outliers}
\label{sec:subl-lift-expect}

The lift expectation involves a nonlinear transformation of the support
function of $X$, and so $\hat{Z}_{X+Y}$ is not necessarily equal to
$\hat{Z}_X+\hat{Z}_Y$. As the following result shows, the lift
expectation is a convex (with respect to the conventional set
inclusion) set-valued function of integrably bounded random convex
bodies.

\begin{proposition}
  \label{prop:subl}
  For integrably bounded random convex bodies $X$ and $Y$, 
  \begin{displaymath}
    \hat{Z}_{tX+(1-t)Y}\subset t\hat{Z}_{X}+(1-t)\hat{Z}_{Y},\quad t\in[0,1].
  \end{displaymath}
\end{proposition}
\begin{proof}
  It suffices to note that 
  \begin{multline*}
    (u_0+th_X(u)+(1-t)h_Y(u))_+\\ \leq t (u_0+th_X(u))_++(1-t)
    (u_0+th_Y(u))_+. \qedhere
  \end{multline*}
\end{proof}

Therefore, sections of the lift expectation 
\begin{displaymath}
  \hat{Z}_X(\alpha)=\{x:\; (\alpha,x)\in\hat{Z}_X\},\quad \alpha\in(0,1),
\end{displaymath}
are convex for Minkowski sums as function of $X$; they can be
interpreted as nonlinear variants of the (linear)
selection expectation. Such expectation of random variables are
extensively studied, see, e.g., \cite{pen04}. If $Z=\{\xi\}$ is a
singleton, then $\alpha^{-1}\hat{Z}_{\{\xi\}}(\alpha)$ is called the
\emph{zonoid-trimmed region} of $\xi$ at level $\alpha$, see
\citep{cas10}. For random vectors, such regions are used to identify
\emph{outliers}, also for random convex bodies they can be used to identify
particularly large sets in the sample, namely those that are not
contained in $\alpha^{-1}\hat{Z}_X(\alpha)$. The parameter $\alpha$ controls the
size of this region; $\alpha$ close to one tends to regard
realisations a little away from the expectation (or sample
mean) as outliers. The convexity property means that if a set is not
an outlier for the sample from the combination $tX+(1-t)X'$ of two
random convex bodies $X$ and $X'$, then it arises as the convex
combination of two non-outliers sampled from $X$ and $X'$.

\subsection{Partial order generated by the lift expectation}
\label{sec:part-order-gener}

It is possible to partially order random convex bodies by the
inclusion of their lift expectations. Theorem~\ref{thr:sf} yields that
$\hat{Z}_X\subset \hat{Z}_{X'}$ if and only if
\begin{displaymath}
  \E(u_0+h_X(u))_+\leq \E(u_0+h_{X'}(u))_+, \quad u_0\in\R,
\end{displaymath}
that is, for all $u\in\R^d$, the random variable $h_X(u)$ (if
integrable) is smaller than $h_{X'}(u)$ with respect to the
\emph{increasing convex order}, see \cite[Th.~1.5.7]{muel:stoy02}. If
$\E X=\E X'$, then the expected support functions coincide and, in
this case, $h_X(u)$ is smaller than $h_{X'}(u)$ with respect to the
convex order, see \cite[Th.~1.5.3]{muel:stoy02}. For singletons (and
the corresponding lift zonoids), this is known as the lift zonoid
order, see \citep{mos02}.

\section{Examples}
\label{sec:prop-exampl-lift}

The following examples show that the lift expectation uniquely
identifies the distribution of random convex bodies with restricted
ranges of possible realisations. As a result, the knowledge of
one-dimensional marginals might suffice to restrore the full
distribution of the support function. 

\begin{example}
  \label{ex:det-shift}
  Let $X=M+\{\xi\}$, where $M$ is a deterministic convex body and $\xi$ is
  an integrable random vector. Without loss of generality assume that
  $\E\xi=0$, otherwise, consider shifted $M$. Then $M=\E X$ is
  uniquely determined by $\hat{Z}_X$, see
  Corollary~\ref{cor:expectation}. Furthermore, the lift expectation
  determines uniquely the distribution of
  $h_X(u)=h_M(u)+\langle\xi,u\rangle$ for all $u$, hence, the
  distribution of $\xi$. 
\end{example}

\begin{example}
  Let $X=\conv\{\xi_1,\dots,\xi_n\}$ be a random polytope determined by
  $n$ i.i.d. random copies of an integrable random
  vector $\xi$. The distribution of $h_X(u)=\max_{1\leq i\leq
    n}\langle \xi_i,u\rangle$ yields the distribution of $\xi$ and so
  the distribution of $X$ is uniquely determined by its
  lift expectation. This is no longer the case if $\xi_1,\dots,\xi_n$
  are not i.i.d. For $d=1$, this example was considered by \cite{cas:men10}.
\end{example}

\begin{example}
  Let $X$ be the sum of segments $[0,\xi_i]$, $i=1,\dots,n$, with
  $\xi_1,\dots,\xi_n\in\R_+^d$, that is, $X$ is a random zonotope in
  $\R_+^d$, and
  \begin{displaymath}
    h_X(u)=\langle u,\xi_1\rangle_++\cdots+\langle u,\xi_n\rangle_+.
  \end{displaymath}
  If the random vectors $\xi_1,\dots,\xi_n$ are i.i.d., then the lift
  expectation identically determines the distribution of $X$. Indeed,
  the distribution of $h_X(u)$ yields the distribution of $\langle
  u,\xi_1\rangle_+$, and so the Laplace transform of $\xi_1$.  
\end{example}

\begin{example}
  Let $X=\{x\in\R^d:\; \langle Q^{-1}x,x\rangle\leq 1\}$ be the
  ellipsoid generated by a positive definite random matrix $Q$, so
  that $h_X(u)=\sqrt{\langle Qu,u\rangle}$. If $Q$ is the random
  diagonal matrix with positive random variables
  $\xi_1,\dots,\xi_d$ on the diagonal, then the distribution of
  $h_X(u)$ yields the distribution of $\sum_{i=1} \xi_i u_i^2$. By
  taking the Laplace transform, it is immediate that the joint
  distribution of $(\xi_1,\dots,\xi_d)$ (and so the distribution of
  $X$) is uniquely determined by the lift expectation of $X$. 
  The distribution of a general (not necessarily diagonal) matrix 
  $Q$ is not uniquely identified.
\end{example}

\begin{example}
  Let $X=\{(\xi_1 x_1,\dots,\xi_d x_d):\; |x_1|+\cdots+|x_d|\leq1\}$
  be the $\ell_1$-ball scaled by the components of
  $\xi=(\xi_1,\dots,\xi_d)\in(0,\infty)^d$. Then
  \begin{displaymath}
    h_X(u)=\max_{i=1,\dots,d} |\xi_i u_i|,
  \end{displaymath}
  and its distributions for each $u\in(0,\infty)^d$ uniquely
  determine the distribution of $\xi$, since
  \begin{displaymath}
    \Prob{h_X(u)\leq t}=\Prob{\xi_1\leq tu_1^{-1},\dots,\xi_d\leq tu_d^{-1}}.
  \end{displaymath}
  Thus, the lift expectation of $X$ uniquely determines the
  distribution of $\xi$ and so that of $X$.  The (nonlifted)
  expectation of the so defined $X$ arises in the theory of
  extreme values; it is called a max-zonoid, see \citep{mo08e}.
\end{example}

\section{Uniqueness of distribution and comonotonicity}
\label{sec:uniq-rand-sets}

In the one-dimensional case, $X=[\xi,\eta]$ is a random segment
determined by two integrable random variables $\xi$ and $\eta$ coupled
so that $\P\{\xi\leq \eta\}=1$. 

\begin{example}
  \label{ex:seg}
  If $X=[\xi,\eta]$, then $\hat{Z}_X$ is a subset of the
  half-plane $[0,\infty)\times\R$ with
  \begin{displaymath}
    h_{\hat{Z}_X}(u_0,u)=
    \begin{cases}
      \E (u_0+u\eta)_+ & u\geq 0,\\
      \E (u_0+u\xi)_+ & u<0,
    \end{cases}
    \quad (u_0,u)\in\R^2.
  \end{displaymath}
  An example of the set $\hat{Z}_X$ is shown on
  Figure~\ref{fig:1}. The upper bound of $\hat{Z}_X$ is the upper
  Lorenz curve of $\eta$, while the lower bound of $\hat{Z}_X$ is the
  lower Lorenz curve of $\xi$. The scaled vertical sections of
  $\hat{Z}_X$ are intervals given by
  \begin{displaymath}
    \alpha^{-1} \{x:\; (\alpha,x)\in \hat{Z}_X\}
    =\big[\inf \E(\zeta\xi),\sup \E(\zeta\eta)\big],\quad \alpha\in(0,1],
  \end{displaymath}
  where the infimum and supremum are taken over random variables
  $\zeta\in[0,\alpha^{-1}]$ such that $\E\zeta=1$. This interval equals
  $[-\mathrm{AVaR}_\alpha(\xi),\mathrm{AVaR}_\alpha(-\eta)]$,
  where $\mathrm{AVaR}_\alpha$ denotes the Average Value-at-Risk at level
  $\alpha$, see, e.g., \citep[Def.~4.43]{foel:sch04}.
\end{example}

The lift expectation of $X=[\xi,\eta]$ yields only the marginal
distributions of $\xi$ and $\eta$, and the uniqueness issue consists
in the existence of a unique measure on the half-plane
$H=\{(x_1,x_2):\; x_1\leq x_2\}$ with given marginals. The joint
distribution of $\xi$ and $\eta$ is unique if it is known that the
random vector $(\xi,\eta)$ is supported by a set $A\subset H$ such
that $A$ intersected with each horizontal or vertical line is a
singleton, equivalently, if one knows the copula of $(\xi,\eta)$. The
following result characterises the uniqueness cases in relation to the
comonotonicity properties of the end-points.

\begin{proposition}
  \label{prop:com-1}
  The distribution of $X=[\xi,\eta]$ is uniquely determined by its
  lift expectation if the end-points $\xi$ and $\eta$ are comonotonic,
  that is, $\xi=f_1(\zeta)$ and $\eta=f_2(\zeta)$ for two monotone
  functions $f_1$ and $f_2$ and a random variable $\zeta$.
\end{proposition}
\begin{proof}
  By Theorem~\ref{thr:sf}, the lift expectation of $X$ uniquely
  determines the marginal distributions of random vector $(\xi,\eta)$,
  and so its joint distribution in view of comonotonicity, see
  \citep{puc:scar10}. 
\end{proof}

The same result holds if $(-\xi,\eta)$ is comonotonic. Without the
comonotonicity assumption, the result does not hold, e.g., $X$ taking
values $[1,3]$ or $[2,4]$ with equal probabilities and $X'$ taking
values $[1,4]$ and $[2,3]$ with equal probabilities share the same
one-dimensional distributions of the support function. However, if the
probabilities of these two values of $X$ are different, then the
distribution of $X$ is uniquely identified, see
Proposition~\ref{prob:dif-values}. The next result follows from the
uniqueness of strongly comonotonic vectors with given marginal
distributions, see \cite[Rem.~3.4]{puc:scar10}.

\begin{proposition}
  \label{prop:com-h}
  The distribution of an integrably bounded random convex compact set
  $X$ in $\R^d$ is uniquely determined by its lift expectation if there exists
  a deterministic function $g:\Sphere\mapsto\{-1,1\}$ such that
  $g(u_1)h_X(u_1),\dots,g(u_n)h_X(u_n)$ form a comonotonic vector for
  all $u_1,\dots,u_n\in\Sphere$ and $n\geq2$. 
\end{proposition}

In particular, Proposition~\ref{prop:com-h} applies if
$h_X(u)=f(u,\eta)$, where $\eta$ is a random variable and the function
$f(u,\cdot)$ is monotone for each $u\in\Sphere$.

\medskip

If $h_X(u)$ is distributed as a random variable $\zeta$ for all $u$
with $\|u\|=1$ (e.g., if $X$ is isotropic), then the lift expectation
of $X$ has the support function $\E(u_0+\zeta\|u\|)_+$. The lift
expectation of $X$ determines the distribution of $\zeta$, but not the
distribution of $X$, for instance, $X$ shares the lift expectation with the
centred ball of radius $\zeta$. 

\begin{example}
  Let $X$ be an isotropic rotation of a deterministic origin symmetric
  convex body $K$. Then $h_X(u)$ has the same distribution for all $u$
  and $\Prob{h_X(u)\leq t}$ for $\|u\|=1$ equals the
  $(d-1)$-dimensional Hausdorff measure (normalised by the area of the
  unit sphere) of the set $\{u\in\Sphere:\; h_K(u)\leq t\}$. Then
  \begin{displaymath}
    \big\{u\in\Sphere:\; h_K(u)\leq t\big\}
    =t\big\{v:\; \|v\|=t^{-1},\; h_K(v)\leq 1\big\}
    =(t\Sphere)\cap K^o,
  \end{displaymath}
  where $K^o=\{v:\; h_K(v)\leq1\}$ is the polar body to
  $K$. Equivalently, the one-dimensional distributions of $h_X$ can be
  retrieved from the Lebesgue measure of $B_r\cap K^o$ for all $r>0$. These values
  do not suffice to retrieve (up to a rotation) a general convex body
  $K$.
\end{example}

\section{Discrete distributions}
\label{sec:discr-distr}

Assume that $X$ is a random convex body that takes a finite number of
values $K_1,\dots,K_N$ with $\Prob{X=K_i}=p_i$, $i=1,\dots,N$. 

\begin{proposition}
  \label{prob:dif-values}
  If $X$ takes a finite number of possible values and their
  probabilities $p_1,\dots,p_N$ are all different, then the
  distribution of $X$ is uniquely determined by its lift expectation.
\end{proposition}
\begin{proof}
  Let $\{u_n,n\geq1\}$ be a countable dense set on $\Sphere$.  The
  lift expectation yields the possible values $t_{n1},\dots,t_{nm_n}$
  of $h_X(u_n)$, where $m_n\leq N$. Since there is a $u_n$ such that
  $h_{K_i}(u_n)$, $i=1,\dots,N$, are all different, we have $\sup_{n\geq1}
  m_n=N$, and so it is possible to recover $N$. The supremum is
  attained, and so one obtains the probabilities $p_1,\dots,p_N$ and
  the values $h_{K_1}(u_n),\dots,h_{K_N}(u_n)$ for some $u_n$. The
  values $h_{K_i}(u)$ can be retrieved by tracing the distribution of
  $h_X(u)$ over $u$ from the unit sphere.
\end{proof}

If the number of realisations is countable, then it is
not possible to recover all individual support functions in some
direction in order to subsequently trace them as in the proof of
Proposition~\ref{prob:dif-values}. If the probabilities are not
different, then the uniqueness fails, as the discussion following
Proposition~\ref{prop:com-1} shows. However, the non-equal
probabilities ensure the uniqueness even if the comonotonicity
condition fails.

\medskip

The \emph{support set} of a convex body $K$ is defined by 
\begin{displaymath}
  H_K(u)=\{x\in K:\; \langle u,x\rangle=h_K(u)\},\quad u\neq 0. 
\end{displaymath}
The support set is a singleton, if $K$ is strictly convex, that is,
its boundary does not contain any nontrivial segment. The support
function of a strictly convex set is differentiable and its gradient
$h'_K(u)$ equals the support point $H_K(u)$ if $\|u\|=1$, see
\cite[Cor.~1.7.3]{schn2}. 

\begin{theorem}
  \label{thr:discrete}
  Assume that the random convex body $X$ has realisations
  $K_1,\dots,K_N$, which are strictly convex and such that
  $H_{K_i}(u)\neq H_{K_j}(u)$ for all $i\neq j$ and all $u\neq0$.
  Then the lift expectation of $X$ uniquely determines its possible
  realisations and their probabilities.
\end{theorem}
\begin{proof}
  There is a $u\in\Sphere$ such that $h_X(u)$ takes $N$ different
  possible values, and this $u$ can be identified from the
  one-dimensional distributions of the support function as the point
  where the number of different values is the largest. Assume that
  these values are $s_1,\dots,s_n$ with probabilities $p_1,\dots,p_n$,
  and note that some or all of the probabilities may be equal.

  Given this $u$, we assign the determined values to the support
  function, so let $h_{K_i}(u)=s_i$, $i=1,\dots,n$. In a neighbourhood
  of $u$, the values of the support functions are still different, and
  it is possible to recover their probabilities there. We let this
  neighbourhood $U$ grow until, at some point $v\in\partial U$,
  $h_{K_i}(v)=h_{K_j}(v)$ for some $i\neq j$. This equality is only
  possible at an isolated point, since otherwise the gradients of
  support functions of $K_i$ and $K_j$ would be equal and these two
  sets would have identical support points. 

  The knowledge of the support functions of $K_i$ and $K_j$ on $U$
  makes it possible to calculate the gradients at $v$ (which are
  different for all $i$ by the assumption) and so identify the support
  functions of $K_i$ and $K_j$ in a neighbourhood of $v$.  Continuing
  this, it is possible to completely identify the support functions of
  all possible realisations.
\end{proof}

\section{Tuples of random convex bodies}
\label{sec:mult-extens}

Let $(X_1,\dots,X_n)$ be an $n$-tuple of integrably bounded random
convex bodies in $\R^d$. Their lift expectation is a convex body in
$\R^{nd+1}$ being the expectation of the convex hull of the origin and
the set $\{1\}\times X_1\times\cdots \times X_n$. By Theorem~\ref{thr:sf},
this lift expectation uniquely determines the distribution of 
\begin{displaymath}
  h_{X_1}(u^{(1)})+\cdots+h_{X_n}(u^{(n)}),
  \quad u^{(1)},\dots,u^{(n)}\in\R^d. 
\end{displaymath}
The following result addresses the case when $X_1=\cdots=X_n=X$. 
Ordering of such lift expectations by inclusion corresponds to the
ordering of $n$-dimensional distributions of $h_X$ in the increasing
positive linear convex order, see \cite[Def.~3.5.1]{muel:stoy02}.

\begin{theorem}
  Let $X$ be an integrably bounded random convex body which almost
  surely contains the origin.  The sequence of lift expectations of
  $(X,\dots,X)$ for $n$-tuples of the same $X$ and all $n\geq1$
  uniquely determines the distribution of $X$.
\end{theorem}
\begin{proof}
  The lift expectation of the $n$-tuple $(X,\dots,X)$ determines the distribution of
  \begin{math}
    h_{X}(u^{(1)})+\cdots+h_{X}(u^{(n)}).
  \end{math}
  By replacing $u^{(i)}$ with $t_iu^{(i)}$ with $t_i\geq0$, we see
  that the distribution of the scalar product of
  $(h_{X}(u^{(1)}),\dots,h_{X}(u^{(n)}))$ and $(t_1,\dots,t_d)$
  is uniquely determined. Since the values of the support
  function are nonnegative, the Laplace transform and so the
  joint distribution of $(h_{X}(u^{(1)}),\dots,h_{X}(u^{(n)}))$ is 
  uniquely determined.
\end{proof}

\begin{example}
  \label{ex:cascos}
  The distribution of $X=[\xi,\eta]$ is uniquely determined by the
  expectation of the convex hull of the origin in $\R^3$ and
  $\{1\}\times[\xi,\eta]\times[\xi,\eta]$. The projections of this
  expected set on the first two coordinates or the first and the third
  coordinates determine the marginal of $\xi$ and $\eta$, while the
  projection on the last two coordinates is the zonoid of
  $(\xi,\eta)$. Note that the lift expectation differs from the lift
  zonoid of $(\xi,\eta)$. 
\end{example}

\section{Numerical example}
\label{sec:numerical-example}

Let $X=[\xi,\eta]$ be a segment on the positive half-line. In
econometrics such segments appear as the data on salary brackets,
given by respondents reluctant to report the exact value of their
salaries. The lift expectation of $X$ is then a convex set in
the plane. The lower boundary of this set coincides with the lower
boundary of the lift zonoid of $\xi$ (also called the generalised
lower Lorenz curve of $\xi$) and the upper boundary stems from the
lift zonoid of $\eta$ (the generalised upper Lorenz curve of $\eta$).

We illustrate the above argument with the US Current Population Survey
2016 freely available from the US Census website. The variable of
interest to us is the personal income variable denoted by
\texttt{ptot\_r} in the dataset. This income variable takes values
from the set $\{1, 2,\dots, 41\}$. Its value $i\in\{1,\dots,40\}$,
means that the personal income of the observed individual lies in the
interval $[a_i,b_i]$, where $b_i-a_i=2499$, and $a_i=b_{i-1}+1$, with
$a_1=0$. We removed from the dataset the individuals with the value
$\texttt{ptot\_r}=41$, meaning that the observed individual earns at
least 100000\$, and so the corresponding interval is unbounded.  Our
final sample includes 132,410 observations.

Figure~\ref{fig:1} shows the estimate for the lift expectation
calculated by taking the mean value of the sets $\{1\}\times[a_i,b_i]$
on the plane. The lower bound (solid black curve) is the Lorenz curve
for the lower bounds of the observed intervals; the upper bound (red
dashed curve) is the upper Lorenz curve for the upper bounds. The (blue
thick) vertical segment with the the $x$-coordinate 1 is the mean
$[27204.40,29450.42]$, obtained by averaging the lower and upper
bounds, see Corollary~\ref{cor:expectation}.

\begin{figure}[htbp]
  \centering
  \includegraphics[scale=0.30]{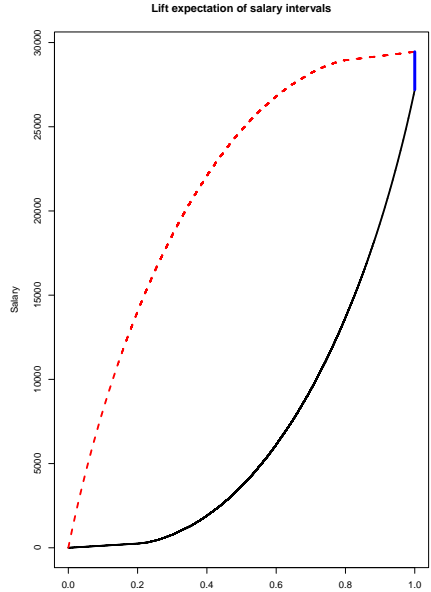} 
  \caption{Estimate of the lift expectation for the interval
    salary data from CPS 2016.
    \label{fig:1}}
\end{figure}

The area of a lift zonoid of a random variable is the Gini mean
difference, which can be used as an inequality index, see,
e.g., \citep{mos02}. In view of this, the area of the lift expectation
provides an upper bound for the Gini inequality index of any random
variable lying in $X=[\xi,\eta]$. In view of this, the area of
$\hat{Z}_X$ may be used as an inequality measure for the case of
interval-valued responses.

\section*{Acknowledgement}

The authors are grateful to the Labex MME-DII and the University
Paris-Saclay (EPEE, Univ-Evry) for hosting a conference where they all
met and initiated this project. The authors are grateful to Ignacio
Cascos for helpful comments and the idea of
Example~\ref{ex:cascos}. GK was supported by the Russian
Science Foundation grant RSF 14-50-00150. IM was supported
by the Swiss National Foundation grants IZ73Z0\_152292 and
200021\_153597.

The authors are grateful to the two anonymous referees and the
associated editor whose comments have led to substantial improvements
of the paper.

\end{document}